\documentclass[12pt]{article}
\usepackage{amsmath,amssymb,amsthm, amsfonts}
\usepackage{mathtools}
\usepackage{hyperref}
\usepackage{graphicx, enumerate}
\usepackage{url}
\usepackage[mathlines]{lineno}
\usepackage{dsfont} 
\usepackage{xcolor}
\definecolor{red}{rgb}{1,0,0}

\definecolor{blue}{rgb}{0,0,1}

\definecolor{green}{rgb}{0,.6,0}

\definecolor{black}{rgb}{0,0,0}

\usepackage{float}
\usepackage{subcaption}
\usepackage{nth}
\usepackage{setspace}
\usepackage{verbatim}

\usepackage{tikz}

\newtheorem{thm}{Theorem}[section]
\newtheorem{cor}[thm]{Corollary}
\newtheorem{lem}[thm]{Lemma}
\newtheorem{prop}[thm]{Proposition}

\newtheorem{obs}[thm]{Observation}

\newcommand{\thr}{\operatorname{th}}
\newcommand{\throtplus}{\operatorname{th_+}}
\newcommand{\thc}{\operatorname{th_\mathit{c}}}

\newcommand{\thp}{\throtplus} 
\newcommand{\thf}{\operatorname{th_\mathit{f}}}
\newcommand{\thx}{\operatorname{th_\mathit{X}}}
\newcommand{\pt}{\operatorname{pt}}
\newcommand{\ptp}{\operatorname{pt_+}}
\newcommand{\ptf}{\operatorname{pt_\mathit{f}}}
\newcommand{\ptx}{\operatorname{pt_\mathit{X}}}

\theoremstyle{definition}
\newtheorem{rem}[thm]{Remark}

\theoremstyle{definition}
\newtheorem{defn}[thm]{Definition}

\theoremstyle{definition}
\newtheorem{ex}[thm]{Example}

\newcommand{\R}{\mathbb{R}}
\newcommand{\ZZ}{\mathbb{Z}}
\newcommand{\N}{\mathbb{N}}

\newcommand{\al}{\alpha}
\newcommand{\bt}{\beta}

\newcommand{\ceq}{\coloneqq}

\newcommand{\ecc}{\operatorname{ecc}}

\newcommand{\dist}{\operatorname{dist}}

\newcommand{\bit}{\begin{itemize}}
\newcommand{\eit}{\end{itemize}}
\newcommand{\ben}{\begin{enumerate}}
\newcommand{\een}{\end{enumerate}}
\newcommand{\beq}{\begin{equation}}
\newcommand{\eeq}{\end{equation}}
\newcommand{\bea}{\begin{eqnarray*}} 
\newcommand{\eea}{\end{eqnarray*}}
\newcommand{\bpf}{\begin{proof}}
\newcommand{\epf}{\end{proof}\ms}
\newcommand{\bmt}{\begin{bmatrix}}
\newcommand{\emt}{\end{bmatrix}}
\newcommand{\ms}{\medskip}

\newcommand{\lc}{\left\lceil}
\newcommand{\rc}{\right\rceil}
\newcommand{\lf}{\left\lfloor}
\newcommand{\rf}{\right\rfloor}

\newcommand{\noi}{\noindent}
\newcommand{\ceil}[1]{\lc #1 \rc}

\newcommand*\patchAmsMathEnvironmentForLineno[1]{%
  \expandafter\let\csname old#1\expandafter\endcsname\csname #1\endcsname
  \expandafter\let\csname oldend#1\expandafter\endcsname\csname end#1\endcsname
  \renewenvironment{#1}%
     {\linenomath\csname old#1\endcsname}%
     {\csname oldend#1\endcsname\endlinenomath}}%
\newcommand*\patchBothAmsMathEnvironmentsForLineno[1]{%
  \patchAmsMathEnvironmentForLineno{#1}%
  \patchAmsMathEnvironmentForLineno{#1*}}%
\AtBeginDocument{%
\patchBothAmsMathEnvironmentsForLineno{equation}%
\patchBothAmsMathEnvironmentsForLineno{align}%
\patchBothAmsMathEnvironmentsForLineno{flalign}%
\patchBothAmsMathEnvironmentsForLineno{alignat}%
\patchBothAmsMathEnvironmentsForLineno{gather}%
\patchBothAmsMathEnvironmentsForLineno{multline}%
}

\newcommand{\s}{\hat s}
\renewcommand{\a}{\alpha}
\renewcommand{\b}{\beta}

\newcommand{\sseq}{\subseteq}

\definecolor{dark}{rgb}{0,0,0}
\definecolor{light}{rgb}{1,1,1}

\makeatletter
\newcommand{\globalcolor}[1]{%
  \color{#1}\global\let\default@color\current@color
}
\makeatother

\title{Throttling processes equivalent to full throttling on trees}
\author{ Michael S. Ross\thanks{Department of Mathematics, Iowa State University,
Ames, IA 50011, USA (msross@iastate.edu)}}

\begin{document}
\maketitle

\vspace{-25pt}\begin{abstract} 
Consider a discrete-time process on a graph $G$ where a set $B$ of initial vertices are chosen to be colored blue (the remainder being white) and then a time step consists of every currently blue vertex forcing all of its neighbors to become blue; this process stops when every vertex of the graph is blue, and the process is called full forcing. The full throttling number of $G$ is then defined to be the minimum sum of the cardinality of $B$ and the number of time steps needed to complete the forcing process. On trees, the full throttling number is equivalent to the throttling numbers of several other graph processes, such as positive-semidefinite zero forcing, the game of cops and robbers, and the distance domination number (alternately, the $k$-radius) of a graph. For all of these, it is known that maximum possible throttling number for a tree on $n$ vertices is somewhere between $1.4502\sqrt{n}$ and $\frac{\sqrt{14}}{2}\sqrt{n}$, with the former exhibited by a family of spiders. After introducing some new ideas and methods for working with throttling on trees, this paper determines the exact full throttling number of all balanced spiders (trees with equal-length paths extending from a center vertex), and proves that their full throttling numbers are bounded above by that of paths of the same order $n$, which are known to have full throttling number $\lc\sqrt{2n}-\frac{1}{2}\rc$.

\end{abstract}

\noi {\bf Keywords} Throttling, full forcing, zero forcing, cops and robbers, $d$-domination

\noi{\bf AMS subject classification} 05C57, 05C15, 05C50

\section{Introduction}\label{sintro} 

The study of throttling graph processes and quantities has its origins as a question about zero forcing \cite{BY}.  Zero forcing is a process used to bound minimum rank/maximum nullity problems from linear algebra and spectral graph theory \cite{MR}, and arose independently in the control of quantum systems \cite{graphinfect, Sev}. The zero forcing number has also been shown \cite{yang2012fast} to be equivalent to fast mixed graph searching. Further, zero forcing appears as a subprocess in the study of power domination \cite{BH Pdom, Liao, ZKC Pdom}, which itself is used as a model for the placement of phase measurement units of electrical networks \cite{Haynes2, Haynes1}. Purely in terms of graph theory, the zero forcing process involves selecting some initial vertices to color blue, and then that blueness can spread through the graph under specific conditions, in discrete time steps. Thus the most natural problems that arise are to determine the smallest number of initial vertices that will eventually color the entire graph blue (this is the \emph{zero forcing number} of the graph), and to determine how much time the process will take with the smallest number of initial vertices necessary (this is the \emph{zero forcing propagation time} of the graph \cite{PT}). Zero forcing throttling seeks to strike a balance between these by first noting that a larger initial set of vertices causes the process to terminate sooner, and then finding an optimal solution that minimizes the sum of the number of initial vertices and the propagation time of the zero forcing process. 

This concept of balancing initial resources against a related parameter has since been applied to other graph processes, and was first extended to positive semidefinite zero forcing \cite{PSDT} which is a process used bound positive semidefinite minimum rank/maximum nullity problems. Subsequently, throttling was applied to the game of cops and robbers  \cite{CRT}, which has applications to the coordination of mobile autonomous agents \cite{ilcinkas}, routing reconfiguration in networks \cite{cs}, and graph decompositions \cite{bs}.

 We now present a simplified forcing process called \emph{full forcing}, and the related throttling process, which will be called \emph{full throttling}.  Let $G$ be a simple graph,  $B\sseq V(G)$ be the current set of blue vertices, and $W$ the current set of white vertices. Then, in a given time step, the \emph{full forcing color change rule} colors $w\in W$ blue whenever $w$ is the neighbor of some $v\in B$, in which case we say that $v$ {\em forces} $w$ and write $v\rightarrow w$. If multiple vertices are capable of forcing $w$, a choice of forcing vertex is made. Ultimately, this does not affect the contents of the sets $B^{(k)}$ defined below. The full forcing process begins with an initial set of blue vertices, $B^{(0)}=B\subseteq V(G)$, with all other vertices being white. The full forcing color change rule is applied iteratively, and the set $B^{(k)}$ is defined to be the set of all the vertices that $\bigcup_{i=0}^{k-1} B^{(i)}$ can force independently; the collection of forces that color the vertices in $B^{(k)}$ blue are said to occur during the \emph{$k^{\text{th}}$ time-step}. If this process eventually colors all of $V(G)$ blue, the set $B^{(0)}$ is said to be a \emph{full forcing set  of $G$}, and the least $k$ such that $\bigcup_{i=0}^{k} B^{(i)}=V(G)$ is the {\em  full forcing propagation time of $B$}, denoted by $\ptf(G;B)$. If $B\subset V(G)$ is not a full forcing forcing set, then $\ptf(G;S)=\infty$.  The \emph{full throttling number of $B$} is $\thf(G;B) = |B| + \ptf(G;B)$, and the \emph{full throttling number $G$} is \[ \thf(G)=\min_{B\sseq V(G)} \thf(G;B). \]
 In the event that $\thf(G;B)=\thf(G)$, $B$ is said to be \emph{optimal}, or more specifically an \emph{optimal full throttling set} of $G$. Note that in a connected graph, every nonempty set of vertices is a full forcing set; and thus the interesting questions about full forcing relate to the propagation time and full throttling numbers.
 
  The \emph{distance between vertices} $u,v$, noted $\dist(u,v)$ is the length of the shortest path between $u$ and $v$. Given a set $S\sseq V(G)$ and a vertex $v$, the \emph{distance from $v$ to $S$} is $\dist(v,S)=\min_{u\in S} \dist(u,v)$. The \emph{eccentricity of a set of vertices} $S\sseq V(G)$ is $\ecc(S)=\max_{v\in V(G)}\dist(v,S)$.
  
  The \emph{distance domination number of $G$}, given a distance $d$, is the size of the smallest set $B$ of vertices with $\ecc(B)=d$; this is denoted by $\gamma_d(G)$.  With a shift in perspective, the $k$-radius of a graph is $\operatorname{rad_\mathit{k}}(G)= \min_{S\sseq V,|S|=k} \ecc(S)$. In their respective notations, the throttling numbers of distance domination and $k$-radius are given by $ \min_{d\geq 0}  \gamma_d(G) + d$ and $\min_{k\geq 1}k + \operatorname{rad_\mathit{k}}(G)$.  The throttling of full forcing, distance domination, and $k$-radius all minimize the sum of a number of vertices and the eccentricity of that set of vertices, and so throttling for each of these parameters is equivalent.
 
The throttling numbers and analogous definitions have been stated for many other processes, a handful of which are given below. Note that for variants of zero forcing, the sets $B^{(i)}$ are defined as for full forcing above; using the associated forcing rules to determine which vertices are forced.
\begin{itemize}
\item For (standard) zero forcing, with $v\in B$ (where $B$ is the current set of blue vertices), and $w\in W=V(G)\setminus B$, $v\to w$ whenever $w$ is the \emph{only} white neighbor of $v$. Then, the propagation time and throttling number of $B$ on $G$ are $\pt(G;B)$ and $\thr(G;B)= |B|+\pt(G;B)$ respectively, and the \emph{throttling number of $G$} \cite{BY} is $\thr(G)=\min_{B\sseq V(G)} \thr(G;B)$.  

\item For positive semidefinite zero forcing, we consider the components $W_1,\dots,W_k$ of the induced subgraph $G[V\setminus B]$; then $v\in B$ forces $w\in W_i$ whenever $w$ is the only white neighbor of $v$ in $G[B\cup W_i]$. The PSD propagation time and PSD-throttling number of $B$ are $\ptp(G;B)$ and $\thp(G;B) = |B| + \ptp(G;B)$. The \emph{PSD-throttling number of $G$} \cite{PSDT} is then $\thp(G)=\min_{B\sseq V(G)} \thp(G;B)$.  

\item For the game of cops and robbers, you are given $k$ cops to place on the graph, and then a robber is placed somewhere in the graph. In a round, you move any number of cops to adjacent vertices --winning the game if a cop occupies the same vertex as the robber-- and then the robber can move to an adjacent vertex, trying to evade capture for long as possible. Assuming the robber is always placed optimally and both players move optimally, the $k$-capture time of $G$ is $\operatorname{capt_\mathit{k}}(G)$, the minimum capture time across all choices of $k$ cops. Then, the \emph{cop throttling number of $G$} \cite{CRT} is $\thc(G)=k+\operatorname{capt_\mathit{k}}(G)$.
\end{itemize}

A throttling process is said to \emph{have full throttling} whenever it is known to be equivalent to full throttling. As noted above, distance domination and $k$-radius throttling always have full throttling. Positive semidefinite zero forcing has full throttling on trees, as every white neighbor of a blue vertex must exist in its own unique component, so every blue vertex forces all of its neighbors each round, and thus is identical to full forcing. In \cite{CRT} it was shown that the cop throttling is equivalent to  full throttling on chordal graphs. 

Thus, on trees, \[ \thf(G)  = \thc(G) = \thp(G). \]
Note that these equalities are not true in general, as it is shown in \cite{CRT,PSDT} that across all graphs, \[\thf(G) \leq \thc(G) \leq \thp(G) \leq \thr(G),  \] with specific examples of graphs where the adjacent throttlings differ.

Except where specifically noted, the rest of this paper will discuss full throttling and those processes which have full throttling on trees, so the notation $\thf(G)$ will be used. It was shown in \cite{PSDT} that full throttling is subtree monotonic on trees, and that for paths of order $n$, $\thf(P_n) =\left \lceil \sqrt{2n}-\frac{1}{2} \right \rceil$. It was also shown in \cite{CRT} that if $T$ is a tree with the highest full throttling number among all trees of order $n$, then $\thf(P_n) \leq \thf(T) \leq 2\sqrt{n}$, with only a couple specific trees known to have $\thf(T)=\thf(P_n)+1$, and no known examples of trees with a higher full throttling number. It has since been shown in \cite{Jesse} that there is a family of trees $T$ of order $n$ with $ 1.45\sqrt{n}\leq\thf(T)$. 
These trees are all examples of \emph{spiders}, which are trees that have exactly one vertex with degree higher than 2.  Spiders are usually described in terms of lengths of their legs; e.g. $S(7,6,2)$ is a tree on 16 vertices, with one \emph{center} vertex adjacent to three disjoint paths of orders 7, 6, and 2.  A \emph{balanced spider} is one in which every leg has the same length, and is generally noted by $T_{\a,\b}=S(\b,\b,\dots,\b)$, where the spider has $\a$ legs, each of length $\b$. Consequently, this paper defines a \emph{super-spider} as any spider which has a full throttling number higher than that of the same-order path. The upper bound was also improved in \cite{Jesse}, where it was noted that $\thf(T)\leq \sqrt{14n}/2$ for any tree $T$ of order $n$, and that for all spiders $S$ of order $n$, $\thf(S)\leq \sqrt{3n}$.

In Section \ref{stools}, we show that full throttling is monotonic for connected minors on trees, define a framework to extend any throttling process to apply to weighted graphs, and present a method for computing the full throttling number of highly symmetric graphs, called \emph{concentration}.  
In Section \ref{spiders} we show that 
\(\thf(T_{\a,\b})=1 + \al \lf \sqrt{\frac{2\bt + \al + 1}{4 \al}} \rf + \lc \frac{\bt - \s}{2\s+1} \rc, \) and prove that there are no balanced super-spiders.

\section{New Tools for Full Throttling on Trees}\label{stools}

In this section we present a useful fact about the full throttling numbers of paths, strengthen the monotonicity results of \cite{PSDT}, and introduce a generalization of throttling processes on (vertex) weighted graphs; which is then used in a new technique for computing the full throttling number of highly symmetric trees, by first reducing them to smaller weighted trees.

\begin{lem}\label{triLem}
Let $t\in\ZZ^+$. Then, the longest path with throttling number $t$ is $P_{n_t}$, where $n_t$ is the $t$-th triangle number. Consequently, $\thf(P_n)=\lc\sqrt{2n+\frac{1}{4}}-\frac{1}{2}\rc$.
\end{lem}

\begin{proof}
Let $n_t$ be the largest integer for which $\thf(P_{n_t})= \lc\sqrt{2n_t}-\frac{1}{2}\rc = t$. Then, $n_t$ is the largest integer such that $\sqrt{2n_t} \leq t +\frac{1}{2}$. By squaring both sides and solving the resulting quadratic, one can see that $n_t = \frac{t(t+1)}{2}$. Thus, $n_t$ is the $t$-th triangle number, and consequently the throttling number of $P_n$ is the ceiling of the inverse triangle number of $n$ .
\end{proof}

On its own, Lemma \ref{triLem} may appear to be no more than a curiosity. However, this variant of $\thf(P_n)$ can lead to some elegant simplification when used alongside other throttling formulae, as in the proof of Lemma \ref{slopes}. Further, the triangle numbers will appear once again when we use them to construct a family of super-spiders in Proposition \ref{triSpid}.

It was established in \cite{PSDT} that full throttling (there called PSD-throttling) is subtree monotonic. We extend full throttling monotonicity to all connected minors.

\begin{obs}\label{cnMinor}
Any connected minor of a tree $T$ can be created using only edge contractions.
\end{obs}


\begin{thm}\label{cMinor}
Let $T$ be a tree, and $T'$ be a connected minor of $T$. Then,
\[ \thf(T') \leq \thf(T). \]
That is, full throttling is connected minor monotonic for trees.
\end{thm}

\begin{proof}
 We need consider only edge contraction by Observation \ref{cnMinor}. Let $uv\in E(T)$, $B \subseteq V(T)$, and $B'$ be the image of $B$ under the edge contraction $T/uv$. That is, $B'$ contains $B\setminus\{u,v\}$, and contains the new vertex  if and only if at least one of $u$ or $v$ are in $B$. Thus, $|B'|\leq |B|$. Next, consider the propagation of $B$ through $T$. If that process forces through edge $uv$, then all subsequent forces in that component will occur one time-step sooner in the propagation of $B'$ through $T/uv$. If the process does not force through edge $uv$, $\ptf(T/uv ; B')\leq\ptf(T;B)$.
Thus, for all $uv\in E(T),$ \(\thf(T/uv)\leq\thf(T). \)
\end{proof}

We now provide a natural extension of the throttling of \emph{any} forcing process to allow for (vertex) weighted graphs. In fact, this provides a blueprint for any process that involves selecting an initial set of vertices from the weighted graph. We'll refer to the generic processes as $X$-forcing and $X$-throttling, and will use $\ptx$ and $\thx$ appropriately.

\begin{defn}
Let $(G,w)$ be a weighted graph, where $w:V(G)\to\R^+$ is a weight function on the vertices of $G$, and let $B\subseteq V(G)$ be an $X$-forcing set of $G$. Define $w(B)=\sum_{v\in B}w(v)$. Then, 
\[\thx(G,w;B) = w(B) + \ptx(G;B) \]
and the \emph{$X$-throttling number} of $(G,w)$ is  
\[\thx(G,w) = \min_{B\subseteq V(G)}\thx(G,w;B).  \]
In the event that $\thx(G,w;B)=\thx(G,w)$, $B$ is said to be an \emph{optimal} ($X$-throttling) set for $(G,w)$.\end{defn}

Note that this method of throttling on weighted graphs is also a generalization of \emph{weighted $X$-throttling} on unweighted graphs, defined for zero forcing and positive semidefinite zero forcing in \cite{BY,PSDT} as
  \[\thf^\omega(G) = \min_{B\subseteq V(G)} \left(\omega|B| + \ptf(G;B)\right),\]  wherein the ``weighting" takes the form of a scalar $\omega$ multiplied by the size of $B$, rather than weights on individual vertices. 

\begin{obs}\label{specCase}
When the weight function $w$ is constant, i.e., $w(v)=\omega$ for all $v\in V(G)$, $X$-throttling of the weighted graph $(G,w)$ is equal to $\omega$-weighted $X$-throttling of the unweighted graph $G$:
$\thx(G,w) =\thx^\omega(G)$.
When the weight function is identically one, the result is ordinary $X$-throttling.
\end{obs}

Next, we define a method by which we can use full throttling on weighted trees to simplify throttling on unweighted graphs, whenever the process being throttled is equivalent on trees to full throttling.

\begin{defn}
Let $(T,w)$ be a weighted tree and let $v\in V(T)$. Then, the components of $T-v$ are called \emph{branches of $T$ at $v$}. Branches $T_1$, $T_2$ of $T$ at $v$ are called \emph{weight isomorphic} when 

\begin{enumerate}
\item there is an automorphism $\sigma$ of  $T$ such that for all $x\in V(T)\setminus(V(T_1)\cup V(T_2))$, $\sigma(x)=x$, $\sigma(V(T_1))=V(T_2)$, and $\sigma(V(T_2))= V(T_1)$, and
\item for all $x \in V(T_1)$, $w(x)=w(\sigma(x))$.
\end{enumerate}
In this case $\sigma$ is called a {\em weight isomorphism}. Given a set  
$T_v=\{T_1,\dots,T_\ell\}$ of pairwise weight isomorphic branches of $T$ at  $v\in V(T)$ with weight isomorphisms $\sigma_i$ between $V(T_i)$ and $V(T_1)$ for $i=2,\dots,\ell$, a set of vertices $B$ is \emph{weight isomorphic with respect to $T_v$} whenever $B\cap V(T_i)=\sigma_i(B\cap V(T_1))$  for $i=2,\dots,\ell$. \end{defn}

\begin{thm}\label{sym}
Let $(T,w)$ be an integer weighted tree. Suppose that 
$T_v=\{T_1,\dots,T_\ell\}$ is a set of pairwise weight isomorphic branches of $T$ at  $v\in V(T)$ with weight isomorphisms $\sigma_i$ between $V(T_i)$ and $V(T_1)$ for $i=2,\dots,\ell$, and that $w(v)=1$. 
Then there is an optimal full forcing set $B$ for $(T,w)$ that is weight isomorphic with respect to $T_v$. 
\end{thm}

\begin{proof}
Let $B$ be an optimal full forcing set of $(T,w)$, and define $B_i=B\cap V(T_i)$ for $i=1,\dots,\ell$.  There are two cases, depending on the role of $v$.

First, suppose $v\in B$ or $v$ can be forced by a vertex not in $\cup_{i=1}^\ell V(T_i)$ in at most $\ptf(T;B)$ time-steps. In particular, this means the full forcing process happens independently in each $T_i$. Without loss of generality, $w(B_1)\le w(B_i)$ for $i=2,\dots,\ell$.  Since forcing in all branches concludes in at most $\ptf(T;B)$ time-steps, if $w(B_1)<w(B_i)$ we could replace $B_i$ by $\sigma_i(B_1)$ and  $B$ was not optimal.  Thus, $w(B_i)=w(B_1)$ for $i=2,\dots,\ell$, and we can replace $B_i$ by $\sigma_i(B_1)$ to get an optimal full forcing set that \emph{is} weight isomorphic with respect to $T_v$.

Next we consider the case  where $v$ is forced by a vertex in some $T_k$, where $1\leq k \leq \ell$. Let $H_i=T[V(T_i)\cup\{v\}]$ for $i=1,\dots,\ell$. We may assume, without loss of generality, that $\{B_i \mid i\neq k  \}$ is weight isomorphic with respect to $\{T_i \mid i\neq k  \}$. Thus, for the remainder of this proof, assume $i\in\{1,\dots,\ell\}$  with $i \neq k$. If $w(B_k)\le w(B_i)$, we could replace $B_i$ by $\sigma_i(\sigma_k^{-1}(B_k))$ -contradicting the optimality of $B$- so we assume $w(B_k)>w(B_i)$. We may also assume $\ptf(H_k;B_k)<\ptf(H_i;B_i)$, or else replacing $B_k$ with $\sigma_k(\sigma_i^{-1}(B_i))$ also contradicts the optimality of $B$. Now, consider $B' = (B\setminus B_k)\cup \sigma_k(\sigma_i^{-1}(B_i)) \cup \{v\}$. Clearly, $\ptf(H_i; B_i \cup \{v\})\leq \ptf(T;B)$, as any forcing caused by $v$ under propagation from $B$ now occurs sooner. Further, since $w(B_k)>w(B_i)$ and $w(v)=1$, $w(B')\leq w(B)$. Thus, $B'$ is an optimal full forcing set for $(T,w)$ which is weight isomorphic with respect to $T_v$.
\end{proof}

It should be noted that there are cases without the condition ``$w(v)=1$", which do not have a weight isomorphic optimal set, as demonstrated with the following example.

\begin{ex}
Consider the balanced spider $T=T_{3,7}\ceq S(7,7,7)$ with center vertex $c$. Let $A=\{ x\in V(T) \mid \dist(c,x)=1\}$,
$B=\{ x\in V(T) \mid \dist(c,x)=5\}$, and suppose $T$ is given weight function
\[w(x)=\begin{cases}
1 & \text{if } \dist(c,x)\in \{1,5\} \\
10 & \text{otherwise}.
\end{cases}\]
First, note that $\thf(T;A)= w(A) + \ptf(T;A)= 3 + 6 =9$, and thus no starting set containing a weight 10 vertex can be optimal. Thus, the only weight isomorphic starting sets that \emph{could} be optimal are $A$, $B$, and $A\cup B$, with $\thf(T;B)=8$, and $\thf(T;A\cup B)=8$. However, if our starting set is $B\cup\{a\}$ where $a\in A$, we get \[\thf(T;B\cup\{a\}) = 4 + 3 = 7. \]
Thus, \emph{no} weight isomorphic set is optimal.
\end{ex}

This construction uses a low cost vertex $v$ near the center $c$ to force through to vertices in other branches, thereby reducing the overall propagation time. However, if $w(c)=1$, then replacing $v$ with $c$ cannot increase the cost, as $w(v)$ is a positive integer. Further, the time need for $v\to c$ is the same as $c\to v$, and the vertices in other branches that $v$ was forcing get forced from $c$ sooner, meaning propagation in all other branches finishes in at most the same amount of time.

\begin{defn}\label{symB}
Let $(T,w)$ be an integer weighted tree. Suppose $\{T_1,\dots,T_\ell\}$ is a maximal set of pairwise weight isomorphic branches of $T$ at vertex $v\in V(T)$, such that $w(v)=1$.
A {\em single concentration of $T$ at $v$} is the weighted tree $T'=T-\{T_2,\dots,T_\ell\}$ with weight function
\[w'(x) = \begin{cases}
\ell w(x)  &\mbox{if }x\in{V(T_1)} \\
w(x) &\mbox{otherwise.}
\end{cases}\]
Each graph formed by one or more iterations of this process is called a \emph{concentration} of $T$.
\end{defn}

\begin{thm}\label{gconc}
Let $(T,w)$ be an integer weighted tree, and let $(T',w')$ be a concentration of $T$. Then,
\[\thf(T)=\thf(T').  \]
\end{thm}

\begin{proof}
Let $T_v=\{T_1,\dots,T_\ell\}$ be the pairwise weight isomorphic branches of $T$ that are concentrated in $T'$. Notice that each set vertices in $T$ that is weight isomorphic with respect to $T_v$ corresponds to exactly one set of vertices in $T'$. 
By Theorem \ref{sym}, there is an optimal set $B$ for $T$ that is weight isomorphic with respect to $T_v$. Let $B'$ be the subset of $V(T')$ corresponding to $B$. 
Clearly, $w(B)=w'(B')$ and $\ptf(T;B)=\ptf(T';B')$, so $\thf(T)\geq\thf(T')$.

Similarly, let $B'$ be an optimal full throttling set of $T'$, and $B$ be the corresponding set of vertices in $T$, which is weight isomorphic with respect to $T_v$. Again, $w(B)=w'(B')$ and $\ptf(T;B)=\ptf(T';B')$, so $\thf(T)\leq\thf(T')$
\end{proof}
The 
concentration approach can simplify  proofs and computations of the full throttling number, especially those with a high degree of symmetry.  

\begin{ex}
Consider $T$, a full binary tree of height $h$.  Suppose $w(v)=1$ for all $v\in V(T)$, and let $c$ denote the root vertex of $T$. Then, consider each vertex at distance $h-1$ from $c$. Each has a weight of 1, and has two weight isomorphic branches (just leaves). Performing a concentration then merges each leaf pair, doubling the cost of the vertices in each branch. Then one can move to the vertices at distance $h-2$ from $c$, each of which has weight 1, and two weight isomorphic branches which are paths. Again, concentrating these paths doubles the weights of the merged vertices. Iterating this concentration process towards $c$ thus results in a path of length $h$, with a weight sequence $2^0,2^1,2^2,\dots,2^h$.
Thus by Theorem \ref{gconc} it's easy to see that $\thf(T')=h+1$ by choosing the vertex that costs only 1, and thus $\thf(T)=h+1$ by choosing the center vertex.
\end{ex}

\section{Spiders}\label{spiders}
In \cite{CRT}, Breen et al.~give an algorithm  that constructs, for any tree $T$,  an initial coloring set $B\subseteq V(T)$ such that  $\thf(T;B)\leq2\sqrt{n}$. Since \cite{PSDT} noted that all paths have a full throttling number approximately $\sqrt{2}\sqrt{n}$, the authors of \cite{CRT} posed an interesting question: What is the smallest coefficient $\mu$ such that for all trees $T$, asymptotically $\thf(T)\lesssim \mu\sqrt{n}$? Or, which trees have the highest full throttling number across all trees on $n$ vertices, and what is that number?
  It was originally thought by some that balanced spiders might provide a family of examples for which $\thf(T)\approx  \mu\sqrt{n}$ with $\mu>\sqrt 2$.  However, we show in this section that 
this is not possible, after determining the exact value of the full throttling number of a balanced spider.

Recall that the (unweighted) balanced spider with $\a$ legs of order $\b$  is $T_{\alpha,\beta}$; which has $\alpha\beta +1$ vertices. Note that $T_{\alpha,\beta}$ has $\a$ weight isomorphic branches at the center vertex $c$, all of which are paths of order $\b$. Thus, $T_{\a,\b}$ can be concentrated to a weighted path of order $\b + 1$, wherein one end vertex (which inherits the label $c$) has weight one, and all other vertices have weight $\a$.

\begin{obs}\label{time}
If $s$ vertices in each leg of $T_{\a,\b}$ (i.e. $s$ non-$c$ vertices from the complete concentration) are optimally chosen and c is \emph{not} chosen, the full forcing propagation time is $\lc\frac{\b+1-s}{2s}\rc$. If $c$ \emph{is} chosen, the full forcing propagation time is $\lc\frac{\b-s}{2s+1}\rc$.
\end{obs}

\begin{lem}\label{cent}
Every (unweighted) balanced spider with at least three legs has an optimal full throttling set containing the center vertex.
\end{lem}

\begin{proof} For $\a\ge 3$, $\b, s\ge 1$,  define
\[g(\a,\b,s)=\a s+\frac{\b+1-s}{2s},\] which corresponds to the full throttling number when $s$ vertices from each leg of $T_{\a,\b}$ are chosen and $c$ is not.
Similarly, for $\a\ge 3, \b\ge 1, s\ge 0$,  define
\[h(\a,\b,s)=1+\a s+\frac{\b-s}{2s+1},\] corresponding to the full throttling number when $c$ is chosen in addition to the $s$ vertices chosen from each leg.
It suffices to show that for every triple $(\a,\b,s)$ with $\a\ge 3$, $\b, s\ge 1$, 
\[h(\a,\b,s)\le g(\a,\b,s)\ \mbox{ or }\ h(\a,\b,s-1)\le g(\a,\b,s).\]
Observe that for  fixed $\a$ and $s$, both $h$ and $g$ are linear functions in $\b$.  The slopes are
\bea \frac{d g(\a,\b,s)}{d\b}\Big|_{s}&=&\frac 1{2s}\\
\frac{d h(\a,\b,s)}{d\b}\Big|_{s}&=&\frac 1{2s+1} <\frac 1{2s}\\
\frac{d h(\a,\b,s)}{d\b}\Big|_{s-1}&=&\frac 1{2s-1}>\frac 1{2s},
\eea
and the intercepts are
\bea 
g(\a,0,s)&=&\a s+\frac {1-s}{2s}=\a s-\frac 1 2+\frac 1{2s}\\
h(\a,0,s)&=&1+\a s -\frac{s}{2s+1}=1+\a s -\frac 1 2 + \frac{1}{4s+2}>\a s-\frac 1 2+\frac 1{2s}\\
h(\a,0,s-1)&=& 1-\a+\a s -\frac{s-1}{2s-1}=1-\a+\a s -\frac 1 2 -\frac{1}{4s-2}< \a s-\frac 1 2+\frac 1{2s}.\\
\eea

Fix $\a$ and $s$.  Define $b_0$ to be the value of $\b$ for which $h(\a,\b,s)=g(\a,\b,s)$, so $b_0=-1 + s + 4 s^2$.
Then, $h(\a,\b,s)\le g(\a,\b,s)$ for $\b\ge b_0$.  Once we show that $h(\a,b_0,s-1)\le g(\a,b_0,s)$, it follows that $h(\a,\b,s-1)\le g(\a,\b,s)$ for $\b\le b_0$, completing the proof.
\[ g(\a,b_0,s)-h(\a,b_0,s-1)=\frac{\a (2 s-1)-4 s+1}{2 s-1}\]
Since $2s-1\ge 1$ it suffices to show that ${1+\a (2 s-1)-4 s}\ge 0$. 
Since $\a\ge 3$ and $s\ge 1$,
\[
1+\a (2 s-1)-4 s\ge 1+3 (2 s-1)-4 s
=1  + 6 s-3-4s
=2s-2
\ge0.\qedhere\]
\end{proof}

\begin{thm}\label{thm:bal-spider-bd}
For the balanced spider $T=T_{\al,\bt}$ with $\a\geq3$, $\thf(T)=1 + \al \s + t$ where
\[\s = \lf \sqrt{\frac{2\bt + \al + 1}{4 \al}} \rf \text{\hspace{2em}and\hspace{2em}} t = \lc \frac{\bt - \s}{2\s+1} \rc= \lc \frac{\bt +\frac 1 2}{2\s+1}-\frac 1 2 \rc. \]
\end{thm}

\begin{proof}
By Lemma \ref{cent}, there is an optimal full throttling set $B_0$ containing the center vertex. Let $P_T$ be the concentration of $T$ at the center. Then the value of $t$ follows from Observation \ref{time}, and thus $\thf(P_T; B_0) = 1 +\al s+ t$ where $s$ is the number of vertices of weight $\a$ (that originally came from the legs).

Now, consider the family of real-valued functions 
\[h(\a,\b,s) = 1 +\al s + \frac{\bt-s}{2s+1},  \]
and note that for fixed $\a\geq 3$ and $s$ these are linear functions of $\beta$.

Observe that the sequence of intercepts $\{h(\a,0,s) \}_{s\in \N}$ is strictly increasing in $s$, and that the sequence of slopes $\{h'(\a,\b,s) \}_{s\in \N}$ is strictly decreasing, but is always positive. We consider the sequence $\{\beta_s\}_{s=0}^\infty$, where $\beta_s$ is the value for which $h(\a,\b_s,s-1)=h(\a,\b_s,s)$, i.e. the point at which increasing from $s-1$ to $s$ vertices will not raise and may lower the full throttling number, which is also the values of $b\in \R$ at which the linear functions $h$ intersect.

\begin{eqnarray}
h(\a,\bt_s,s-1)&=& h(\a,\bt_s,s)\nonumber \\
1 + \al (s-1) + \frac{\bt_s - (s-1)}{2(s-1)+1} &=& 1 + \al s + \frac{\bt_s - s}{2s+1}\nonumber \\
\frac{\bt_s - s+1}{2s-1} &=& \al + \frac{\bt_s - s}{2s+1}\nonumber \\
\frac{\bt_s - s+1}{2s-1} &=& \frac{2\al s + \al +\bt_s - s}{2s+1}\nonumber \\
2\bt_s s - 2s^2+ 2s + \bt_s - s+1 &=& 4\al s^2 + 2\al s + 2\bt_s s -2s^2  -2\al s - \al -\bt_s + s\nonumber \\
2 \beta_s &=& 4\al s^2 -\al -1\nonumber \\
\beta_s &=& 2\al s^2 -\frac{\al+1}{2}. \label{b}
\end{eqnarray}

Solving \eqref{b} for $s$ and taking the floor then gives the optimal choice for $\s$, given any $\b$.
\[\s = \lf\sqrt{\frac{2\bt + \al + 1}{4 \al}}\rf. \qedhere\] 
\end{proof}

Here, we define a continuous variant of the full throttling number for balanced spiders, which will be used in the next section. Let \[t_S(\a,\b) = 1 +\a\s + \hat t  \]
where $\s$ is as defined in Theorem \ref{thm:bal-spider-bd}, and $\hat t = \frac{\b + \frac{1}{2}}{2\s +1}-\frac{1}{2}$. Note that $\hat t$ is obtained by removing the ceiling from $t$ in Theorem \ref{thm:bal-spider-bd}.

\begin{cor}\label{ts}
The functions $t_S(\a,\b)$ are continuous in $\b$, and $\thf(T_{\a,\b})=\ceil{t_S(\a,\b)}$.
\end{cor}

\begin{proof}
Note that the second statement follows immediately from the fact that $\a$ and $\s$ are integers. For the first, note that for all $\b$, $t_S(\a,\b)=h(\a,\b,\s)$, and so $t_S(\a,\b) = \min_{s\in\N} h(\a,\b,s)$. Finally, recall that when the optimal value of $s$, changes to $s+1$, it is at the 
$\b$ for which $h(\a,\b,s)=h(\a,\b,s+1)$, and thus $t_S(\a,\b)$ must be continuous.
\end{proof}

In \cite{CRT}, it is shown that the spider $S(4,3,2)$ has a higher full throttling number than the path of the same order. Specifically, $\thf(S(4,3,2))=5=1+\thf(P_{10})$. A computer search of small spiders shows that this is the smallest spider whose full throttling number exceeds that of the path of the same order.  This search, which is described in Appendix 1, also produced several thousand spiders that have full throttling numbers one more than that of the path of the same order.  For example, the full throttling numbers of next few smallest such spiders  $\thf(S(5,4,3,2))=\thf(S(5,4,4,1))=\thf(S(6,4,4))=\thf(S(7,4,3))=6>5=\thf(P_{15})$ and 
$\thf(S(6,5,4,4))=\thf(S(7,5,4,3))=7>6=\thf(P_{20})$.

In light of this, we define a \emph{super-spider} as a spider $S$ for which $\thf(S)>\thf(P_{|V(S)|})$. We give a simple infinite family of super-spiders (the triangle spiders) with exactly ``path plus one" full throttling number below.




\begin{prop}\label{triSpid}
Let $t\in \ZZ^+$ with $t\geq 4$. Then, the spider $S=S(t,t-1,\dots,2)$ with $t-1$ legs on $n=\frac{t(t+1)}{2}$ vertices has full throttling number $\thf(S)= t + 1 = 1 + \thf(P_n)$.
\end{prop}

\begin{proof}
Let $\ell_k$ be the length $k$ leg of $S$, and let $p=t-i$ be the proposed propagation time of an optimally chosen set $B$ of vertices. First, observe that all legs $\ell_k$ with $k>p$ \emph{must} contain a blue vertex if propagation is going to conclude on time. Next, note that leg $\ell_p$ cannot be fully forced by any vertex in $\ell_{p+1}$, as the distance from $\ell_{p+1}$'s most central vertex and $\ell_p$'s least central vertex is $p+1$. Thus, we must either choose a vertex in $\ell_p$, or choose the center vertex. As the center vertex will guarantee the forcing of all $\ell_k$ with $k\leq p$, this is clearly an optimal choice. Thus we have $\thf(S;B)\geq (i + 1) + p = t+1$. So $\thf(S)\geq t+1$. Note that choosing $p=t$ by just coloring the center vertex gives $\thf(S)\leq t+1$. Then Lemma \ref{triLem} gives $\thf(S)=\thf(P_n)+1$.
\end{proof}

\begin{rem}
As an immediate consequence of Proposition \ref{triSpid}, there is no constant bound on the number of legs a super-spider can have.
\end{rem}

Finally, we show that there are no \emph{balanced} super-spiders. To do so, we will examine the continuous analogues of the full throttling number functions for balanced spiders and paths. $t_S(\a,\b)$ is already defined before Corollary \ref{ts}.  For paths, recall from Lemma \ref{triLem} that $\thf(P_{\a\b+1}) = \ceil{\sqrt{2(\a\b+1)+\frac{1}{4}}-\frac{1}{2}}$. Thus, we define 
\[t_P(\a,\b)\ceq\sqrt{2(\a\b+1)+\frac{1}{4}}-\frac{1}{2}. \] Since $\thf(P_{\a\b+1})$ and $\thf(T_{\a,\b})$ are the respective ceilings of $t_P(\a,\b)$ and $t_S(\a,\b)$, we need only demonstrate that $t_S(\a,\b)\leq t_P(\a,\b)$.

\begin{lem}\label{slopes}
Let $\a\geq3$, and suppose $t_S(\a,\b)\leq t_P(\a,\b)$ for some $\b\geq \b_1$, where $\b_1$ is defined in (\ref{b}). Then $t_S(\a,\b') \leq t_P(\a,\b')$ for all $\b' \geq \b$.
\end{lem}

\begin{proof}
Observe that $t_S$ is locally linear in $\b$ (except at each $\b_s$), whereas $t_P$ is concave down in $\b$. Thus, we need only show that the inequality holds for the values $\b_s$, where each change in slope occurs. We prove this by examining the average rates of change of the respective functions over the intervals $[\b_s,\b_{s+1}]$.
As $t_S(\a,\b)$ is linear on each $[\b_s,\b_{s+1}]$, we know it has slope \(\frac{1}{2s+1}\). Computing the average slope of $t_P$ on each interval is a bit trickier.

\begin{align*}
\frac{t_P(\a,\b_{s+1}) - t_P(\a,\b_s)}{\b_{s+1}-\b_s} &= \frac{\left(\sqrt{2(\a\b_{s+1}+1) + \frac{1}{4}} -\frac{1}{2}\right) - \left(\sqrt{2(\a\b_s+1) + \frac{1}{4}} -\frac{1}{2}\right)}{\b_{s+1} - \b_{s}} \\
  &= \frac{\sqrt{2\a\b_{s+1} + \frac{9}{4}}  - \sqrt{2\a\b_s + \frac{9}{4}}}{\b_{s+1} - \b_{s}} \\
  &= \frac{\sqrt{4\a^2s^2 + 8\a^2s + 3\a^2 + -\a + \frac{9}{4}} - \sqrt{4\a^2s^2-\a^2 -\a +\frac{9}{4}}}{4\a s+2\a}\\
 &= \frac{\sqrt{16\a^2s^2 + 32\a^2s + 12\a^2 -4\a + 9} - \sqrt{16\a^2s^2-4\a^2 -4\a +9}}{4\a (2s+1)}
\end{align*}

With a little algebraic manipulation, we see that the average rate of change for $t_S(\a,\b)$ is less than the average rate of change of $t_P(\a,\b)$ when the inequality
\[4\a \leq  \sqrt{16\a^2s^2 + 32\a^2s + 12\a^2 -4\a + 9} - \sqrt{16\a^2s^2-4\a^2 -4\a +9} \] holds. To establish this condition, $r_{s+1}$ and $r_s$ will be used as shorthand for the two square roots, respectively. Thus we need to show that

\begin{align*}
(4\a+r_s)^2 &\leq r_{s+1}^2 \\
16\a^2 +8\a r_s &\leq r_{s+1}^2 -r_s^2 \\
16\a^2 + 8\a r_s &\leq 32\a^2s + 16\a^2 \\
8\a r_s &\leq 32\a^2 s \\
r_s^2 &\leq 16\a^2 s^2 \\
16\a^2 -4\a^2 -4\a +9 &\leq 16\a^2s^2 \\
9 &\leq 4\a^2 + 4\a
\end{align*}

Since $\a\geq 3$ by hypothesis, the last inequality holds for all balanced spiders. 
\end{proof}

One should note that this inequality is true for all $\b$. However, $\b_0$ is actually negative, and thus has no context within the problem. Hence the initial restriction $\b\geq \b_1$.

\begin{lem} \label{3spid}
For all $\a\geq3$, $t_S(\a,\b_2)\leq t_P(\a,\b_2)$.
\end{lem}

\begin{proof}
Note that $\b_2= 8\a-\frac{\a+1}{2}$. Then we need only show that the difference
\begin{align*}
t_P(\a,\b_2)-t_S(\a,\b_2) &= \sqrt{2\a\b_2+\frac{9}{4}} - \left(1 + 2\a + \frac{\b_2 + \frac{1}{2}}{5} \right) \\
                          &= \sqrt{2\a\left(8\a-\frac{\a+1}{2}\right)+\frac{9}{4}} - \left(1 + \frac{10\a + (8\a-\frac{\a+1}{2}) + \frac{1}{2}}{5} \right) \\
                          &= \sqrt{15\a^2-\a +\frac{9}{4}} - \frac{7}{2}\a -1  
\end{align*}
is non-negative, which it is for all $\a\geq 1$.
\end{proof}

\begin{lem}\label{5spid}
For all $\a\geq 5$, $t_S(\a,\b_1)\leq t_P(\a,\b_1)$.
\end{lem}

\begin{proof}
Note that $\b_1 = \frac{3\a-1}{2}$. Then, we need only show that the difference

\begin{align*}
t_P(\a,\b_1)-t_S(\a,\b_1)   &= \sqrt{2\a\b_1+\frac{9}{4}} - \left(1 + \a + \frac{\b_1+\frac{1}{2}}{3}\right) \\
                            &= \sqrt{2\a \left( \frac{3\a-1}{2} \right)+\frac{9}{4}} - \left(1 + \a + \frac{\left( \frac{3\a-1}{2} \right)+\frac{1}{2}}{3}\right) \\
                            &= \sqrt{3\a^2 -\a +\frac{9}{4}} - \frac{3}{2}\a -1
\end{align*}
is non-negative. which it is for all $\a\geq 5$.
\end{proof}

\begin{thm}
There are no balanced super-spiders.
\end{thm}

\begin{proof}

Assume first that $\a\geq 5$. By Lemma \ref{5spid}, $t_S(\a,\b)\leq t_P(\a,\b)$ for all $\b\geq\b_1$. Further, since $t_P(\a,0)=t_S(\a,0)=1$ for all $\a\geq 3$, $t_P$ is concave down in $\b$, and $t_S$ is linear in $b$ over the interval $[0,\b_1]$, we have that  $t_S(\a,\b)\leq t_P(\a,\b)$ for all $\b\geq 0$. Thus, there are no balanced super-spiders on five or more legs, and we need only demonstrate that there are no balanced super-spiders on three or four legs. 

Now, suppose $\a\in\{3,4\}$. Then $t_S(\a,\b_1)\geq t_P(\a,\b_1)$, and $t_S(\a,\b_2)\leq t_P(\a,\b_2)$. Thus, for $\a\in\{3,4\}$, there is a point $b_1$ in the interval $[0,\b_1]$ where $t_S$ becomes larger than $t_P$, and there is another point $b_2$ in the interval $[\b_1,\b_2]$ where $t_P$ becomes larger than $t_S$ (See Figure 1). Since Lemma \ref{3spid} proves there are no balanced super spiders with $\b\geq \b_2$, any balanced super-spider must have $\a\in\{3,4\}$, and $\b\in (b_1,b_2)$. 

Suppose $\a=3$. Then $\b_1=4$ and  $t_S(3,3)=t_P(3,3)=4$, so $b_1=3$. On the other end, $\b_2=22$ and  $\frac{17}{3}=t_S(3,6)<t_P(3,6)=5.685$,  so $b_2<6$. Thus, the only integer candidates for $\b$ are 4 and 5.
Suppose $\a=4$. Then $\b_1=5.5$ and $t_S(4,5)=t_P(4,5)=6$, so $b_1=5$. On the other end, $b_2=29.5$ and  $7=t_S(4,7)<t_P(4,7) \approx 7.132$, so $b_2<7$. Thus, the only integer candidate for $\b$ is 6.

To summarize, the balanced spiders $T_{3,4}$, $T_{3,5}$, and $T_{4,6}$ are the \emph{only} candidate balanced super-spiders. However, it is easy to verify that each these has a full throttling number \emph{equal} to that of its correlated path, and thus is \emph{not} a super-spider.
\end{proof}

\begin{figure}
    \centering
    \begin{subfigure}[b]{0.45\textwidth}
        \includegraphics[width=\textwidth]{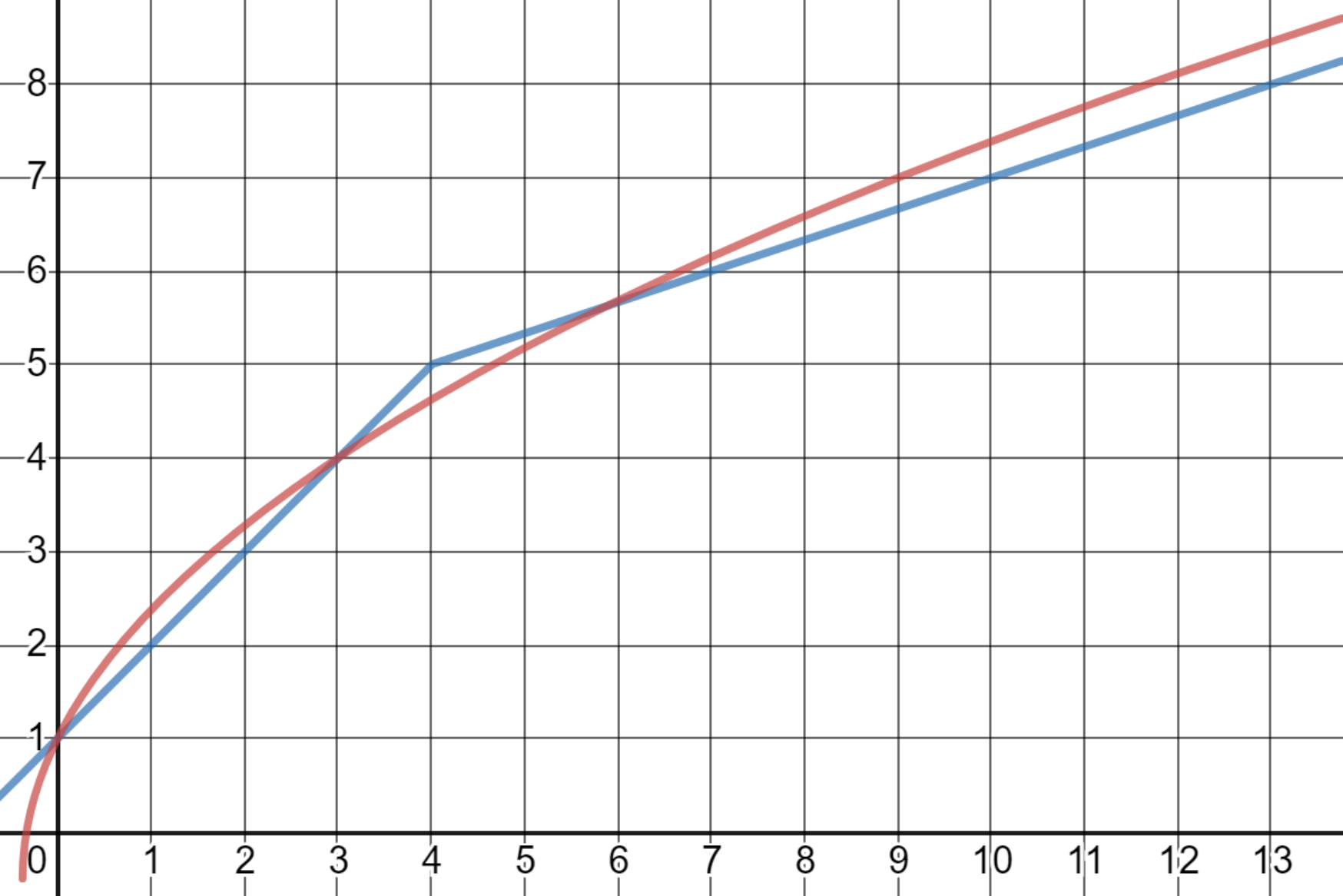}
        \caption{$\a=3$}
        \label{fig:sspid3}
    \end{subfigure}
    ~ 
    \begin{subfigure}[b]{0.45\textwidth}
        \includegraphics[width=\textwidth]{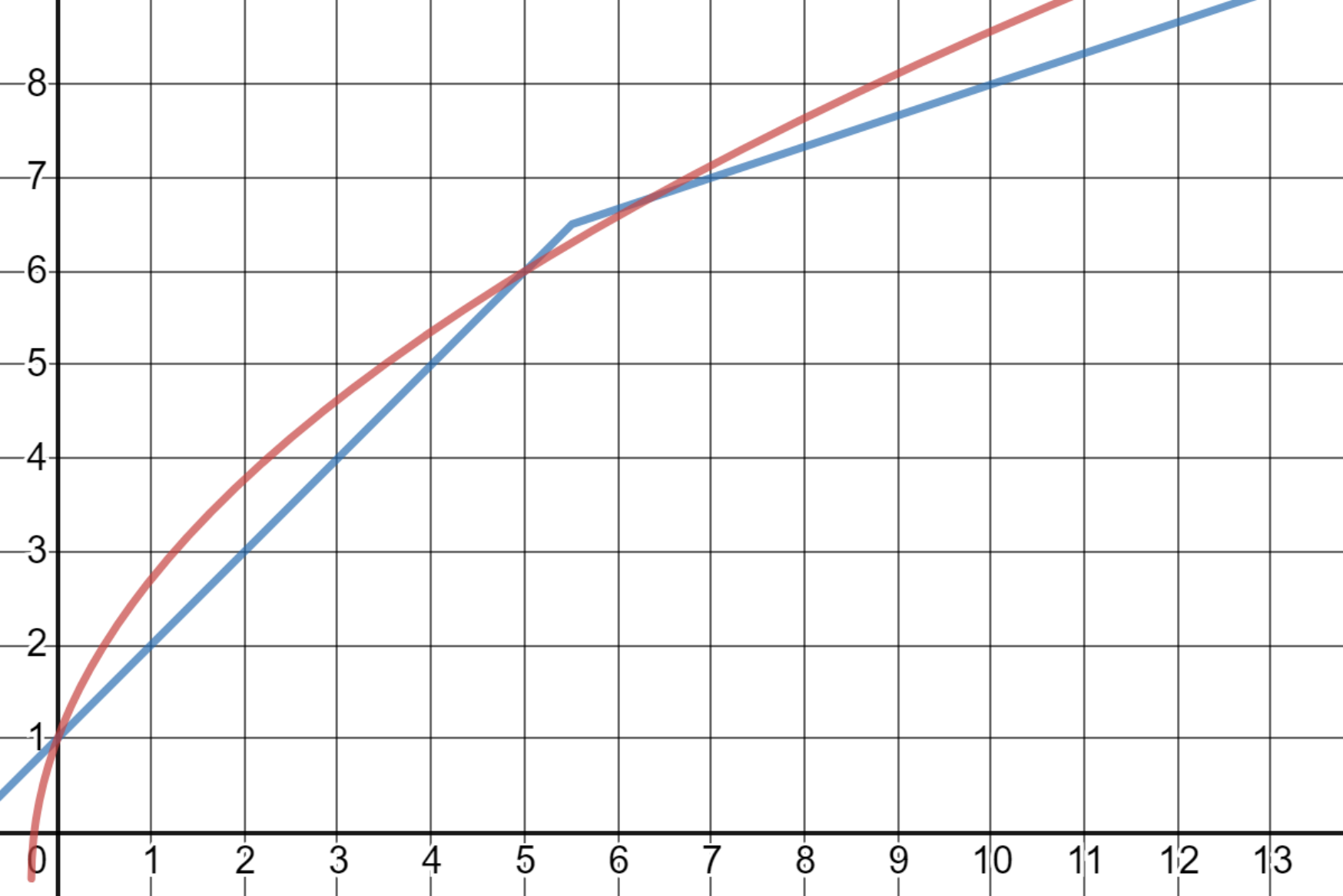}
        \caption{$\a=4$}
        \label{fig:sspid4}
    \end{subfigure}
    \caption{$t_S$ and $t_P$ as functions of $\b$}\label{fig:graphs}
\end{figure}

\appendix
\section{Algorithms, Computations, and Data}

The overall process is as follows. Given $n$, we first compute the full throttling number $t$ of $P_n$. To iterate through all spiders, we observe that the spiders have a one-to-one correspondence with the partitions of the integer $n-1$ that have at least three parts. Next, we iterate through all possible values for $|B|=s\leq \thf(P_n)$. Once the spider and proposed starting size are chosen, the recursive Sage function below determines if the spider can be fully forced within $p=t-s$ time steps.

\vspace{1em}
\hrule
\begin{verbatim}
def spidthrot(partlist, cbool, s, p):
    #partlist = Partition representing the spider
    #cbool = Boolean, true if center is already colored
    #s = remaining number of choices for starting set
    #p = proposed propagation time
    plist=list(partlist)        #Convert Partition to list
    plist.sort()                #Sort legs
    if (not plist) and (cbool or s>0):  
                                #Legs empty and can finish
        return true
    elif(not plist):            #Legs empty but can't finish
        return false
    elif plist and s == 0:      #Legs not empty, can't choose more; s>=0 for rest
        return false 
    elif plist[-1] > 2*p +1:    #Longest needs more than 1, can choose more
        l = plist.pop()
        l = l - (2*p+1)
        plist.append(l)
        return spidthrot(plist, cbool, s-1, p)
                                #Recursive, cover longest, costs 1
    elif plist[-1] == 2*p +1:   #Longest needs full time, w/o center
        plist.pop()
        return spidthrot(plist, cbool, s-1, p)
    elif plist[-1] == 2*p:      #Covering longest includes center
        plist.pop()
        return spidthrot(plist, true, s-1, p)
    elif plist[-1] > p:         #Covering longest cleans up short legs
        l = plist.pop()
        while plist and plist[0]<= 2*p - l:
            plist.pop(0)
        return spidthrot(plist, true, s-1, p)
    else:                       #Have one to spare, and choosing center covers all.
        return true
\end{verbatim}
\hrule
\vspace{1em}

If the spider has full throttling number at most $t$, we move on to the next spider. If not, we run the recursion for $t\leq t' \leq t+k$ (usually $k=1$), to determine the spider's full throttling number, with a special message given if the choice of $k$ is too small.

Below is a table containing most of what is known about super-spiders. For all $1\leq n\leq 74$, $n$ is omitted from the table if there are no super-spiders of order $n$. The smallest super-spiders for each value of $t$ are given as tuples.

\begin{figure}[ht]
\begin{center}
\begin{tabular}{|c|c|c|c|}
\hline
$n$        &  $\thf(P_n)$   & \# of S-Spiders  & Examples \\
\hline
10         &  4             &  1            & $(4,3,2)$\\
\hline
15         &  5             &  4            & $(5,4,3,2),(5,4,4,1),(6,4,4),(7,4,3)$ \\
\hline
20         &  6             &  2            & $(6,5,4,4),(7,5,4,3)$ \\
\hline
21         &  6             &  17           & -Many-  \\
\hline
26         &  7             &  3            & $(7,6,5,4,3),(9,6,5,5),(10,6,5,4)$ \\ \hline
27         &  7             &  17           & -Many- \\
\hline
28         &  7             &  62           & -Many- \\
\hline
33         &  8             &  2            & $(9,7,6,5,5),(10,7,6,5,4)$ \\
\hline
34         &  8             &  19           &  -Many-  \\
\hline
35         &  8             &  77           &  -Many- \\
\hline
36         &  8             &  221          &  -Many \\
\hline
41         &  9             &  5            &  {\begin{tabular}{@{}c@{}}$(9,8,7,6,5,5),(10,8,7,6,5,4),(12,9,7,6,6)$ \\  $(13,8,7,6,6),(13,9,7,6,5)$\end{tabular}} \\
\hline
42         &  9             &  31           &  -Many-  \\
\hline
43         &  9             &  118          &  -Many-  \\
\hline
44         &  9             &  330          &  -Many-  \\
\hline
45         &  9             &  783          &  -Many-  \\
\hline
49         &  10            &  2            &  $(12,9,8,7,6,6),(13,9,8,7,6,5)$\\ 
\hline
50         &  10            &  14           &  -Many-  \\
\hline
51         &  10            &  61           &  -Many-  \\
\hline
52         &  10            &  210          &  -Many-  \\
\hline
53         &  10            &  595          &  -Many-  \\
\hline
54         &  10            &  1399         &  -Many-  \\
\hline
55         &  10            &  2920         &  -Many- \\
\hline
59         &  11            &  4            &  {\begin{tabular}{@{}c@{}}$(12, 10, 9, 8, 7, 6, 6),(13, 10, 9, 8, 7, 6, 5),$ \\ $(15, 12, 9, 8, 7, 7),(16, 12, 9, 8, 7, 6)$ \end{tabular}} \\
\hline
60         &  11            &  32           &  -Many- \\
\hline
61         &  11            &  131          &  -Many- \\
\hline
62         &  11            &  441          &  -Many-  \\
\hline
63         &  11            &  1201         &  -Many-  \\
\hline
64         &  11            &  2803         &  -Many-  \\
\hline
65         &  11            &  5792         &  -Many- \\
\hline
66         &  11            &  10986        &  -Many-  \\
\hline
69         &  12            &  3            & {\begin{tabular}{@{}c@{}}$(15, 12, 10, 9, 8, 7, 7),
(16, 11, 10, 9, 8, 7, 7),$ \\
$(16, 12, 10, 9, 8, 7, 6)$ \end{tabular}} \\
\hline
70         &  12            &  22           &  -Many- \\
\hline
71         &  12            &  104          &  -Many-  \\
\hline
72         &  12            &  380          &  -Many- \\
\hline
73         &  12            &  1123         &  -Many-  \\
\hline
74         &  12            &  2823         &  -Many- \\
\hline

\end{tabular}
\end{center}

\end{figure}

Note that as the full throttling number increases, super-spiders appear sooner (relative to $n_t$). This suggests a potential way to construct a tree with a full throttling number higher than path plus one. For example, $n_{12}=78$, but we have a super-spider (with $\thf(S)=13$) on 69 vertices, $S(15,12,10,9,8,7,7)$. Thus there are 9 vertices one might cleverly place to get a full throttling number of 14.

\end{document}